\documentclass[]{amsart}
\usepackage{latexsym}
\usepackage{amsmath}
\usepackage{amsfonts}
\usepackage{amsthm}
\usepackage{amssymb}
\usepackage[T1]{fontenc}
\usepackage{hyperref}
\usepackage{amsrefs}

\title[permutations of series]{On the $c_0$-equivalence and permutations of series}
\author{Artur Bartoszewicz}
\author{W\l{}odzimierz Fechner}
\author{Aleksandra \'Swi\k{a}tczak}
\author{Agnieszka Widz}
\address{Institute of Mathematics, Lodz University of Technology, ul. W\'olcza\'nska 215, 90-924 \L\'od\'z, Poland}
\email{artur.bartoszewicz@p.lodz.pl}
\email{wlodzimierz.fechner@p.lodz.pl}
\email{aleswi97@gmail.com}
\email{AgnieszkaWidzENFP@gmail.com}

\newtheorem{thm}{Theorem}
\newtheorem{cor}{Corollary}
\theoremstyle{remark}
\newtheorem{ex}{Example}
\theoremstyle{definition}

\newcommand{\R}{\mathbb{R}}

\newcommand{\N}{\mathbb{N}}
\newcommand{\I}{\mathcal{I}}
\newcommand{\eps}{\varepsilon}

\newcommand{\Hn}{\mathrm{Hn}}

\renewcommand{\(}{\left(} \renewcommand{\)}{\right)}

\keywords{$c_0$-equivalence, center of distances, potentially convergent series, von Neumann's theorem, Riemann rearrangement theorem, hypernumber}
\subjclass[2010]{40A05, 40A35}

\begin{document}

\begin{abstract}
Assume that a convergent series of real numbers $\sum\limits_{n=1}^\infty a_n$ has the property that there exists a set $A\subseteq \N$ such that the series  $\sum\limits_{n \in A} a_n$ is conditionally convergent. We prove that for a given arbitrary sequence $(b_n)$ of real numbers there exists a permutation $\sigma\colon \N \to \N$ such that $\sigma(n) = n$ for every $n \notin A$ and $(b_n)$ is $c_0$-equivalent to a subsequence of the sequence of partial sums of the series $\sum\limits_{n=1}^\infty a_{\sigma(n)}$. 

Moreover, we discuss a connection between our main result with the classical Riemann series theorem.

\end{abstract}

\maketitle 

\section{Introduction}

By $\R$ we denote the set of reals and $\N=\{1, 2, \ldots \}$. By an \emph{ideal}, or \emph{set-ideal} to distinguish from an algebraic ideal, we mean a family $\I\subset \mathcal{P}(\N)$ which is closed under  finite {unions} and contains all subsets of {each of its members}. An ideal is \emph{proper} if it does not contain $\N$ and a proper ideal is \emph{admissible} if it contains all singletons $\{n\}$. 
An ideal $\I$ is called \emph{summable} if there exists a sequence of positive numbers $(a_n)$ such that
 $\sum\limits_{n=1}^\infty a_n=+\infty$ and $A\in \I$ if and only if  $\sum\limits_{n \in A} a_n< + \infty$.

Next, for a given ideal $\I$ we say that a sequence $(x_n)$ of real numbers is \emph{$\I$-bounded} if there exists a constant $\lambda>0$ such that the set $\{n : |x_n|>\lambda \}$ belongs to $\I$. The set of all $\I$-bounded sequences is denoted by $\ell^\infty (\I)$ and is equipped with a semi-norm given by
$$\|(x_n)\|^\I_\infty = \inf\{ \lambda >0 : \{n\in\N : |x_n|>\lambda \}\in \I \}.$$
Further, a sequence $(x_n)$ is termed \emph{$\I$-convergent} to some $x_0 \in \R$ if for every $\eps>0$ the set $\{n\in\N : |x_n-x_0|>\eps \}$ belongs to $\I$. The set of all $\I$-convergent sequences is denoted by $c(\I)$ and by $c_0(\I)$ we mean its subspace of all sequences which are $\I$-convergent to $0$.
We say that two sequences $(x_n)$ and $(y_n)$ of elements of a metric space $(X,d)$ are \emph{$c_0$-equivalent} if
$$\lim_{n \to \infty} d(x_n,y_n)=0.$$ This relation is most often considered in the case of $\ell^\infty$ or of a compact space $X$. Below we illustrate both situations by recalling some applications of this notion.

\begin{thm}[Z. Semadeni {\cite{S}*{Theorem 4.2.2, p.77}}]
The space $\ell^{\infty}/c_0$ is isometrically isomorphic to the space of continuous functions on the remainder $\beta \N \setminus \N$ of the \v{C}ech-Stone compactification of {the} discrete space $\N$.
\end{thm}

In \cite{BGW} the authors have generalized this by proving a set-ideal version of the above theorem. Given an admissible ideal $\I\subset \mathcal{P}(\N)$ by $P_\I$ we denote the set of all proper ultrafilters $p$ in $\mathcal{P}(\N)$ such that $\bigcap p = \varnothing$ and the family $p^*$ of the complements of all elements of $p$ contains $\I$. The family $P_\I$ is a closed subset of $\beta\N$ (see B. Balcar, F. Simon \cite{BS}) and therefore it makes sense to consider the space $C(P_\I)$ of all continuous functions on $P_\I$.

\begin{thm}[A. Bartoszewicz, Sz. G\l \c{a}b, A. Wachowicz {\cite{BGW}*{Theorem 1}}]
For every admissible ideal $\I\subset \mathcal{P}(\N)$ the spaces $\ell^{\infty}(\I)/c_0(\I)$ and $C(P_\I)$ are isometrically isomorphic.
\end{thm}

Roughly speaking, the above statement says that in Semadeni's result one can use the seminorm $\|\cdot \|^{\mathrm{Fin}}_\infty $ on the space $\ell^\infty$, where $\mathrm{Fin}$ is the ideal of finite subsets of $\N$, obtaining the same effect.

\medskip

{According to a} well-known theorem of von-Neumann, two sequences $(x_n)$, $(y_n)$ of elements of a compact metric space have the same set of limit points if and only if $(x_n)$ is $c_0$-equivalent to some permutation of $(y_n)$. Recently W. Bielas, Sz. Plewik and M. Walczy\'{n}ska \cite{BPW}*{Theorem 2.1} have generalized one implication of the von-Neumann's theorem using the so-called ``back-and-forth'' method. 
In a paper {by} M. Banakiewicz, A. Bartoszewicz {and} F. Prus-Wi\'sniowski \cite{BBPW} the reverse implication to this theorem was proved.
Given a metric space $(X,d)$ the \emph{center of distances} of $X$ is defined as
$$ Z(X) = \{ \alpha \geq 0 : \forall {x \in X}, \exists {y\in X} \  \text{such that} \  d(x,y) = \alpha \}.$$ 

\begin{thm}[ {\cite{BBPW}*{Theorem 1.2}}]
Assume that $(X,d)$ is a compact metric space, $(x_n), (y_n)$ are sequences of elements of $X$ which have the same set of limit points $A$.
Then, for a given $\alpha \geq 0$ there exists a permutation $\sigma\colon \N \to \N$ such that $$\lim_{n \to \infty} d(x_n,y_{\sigma(n)})=\alpha$$   if and only if $\alpha \in Z(A)$. 
\end{thm}

From the above statement one can easily derive the less trivial implication of the von-Neumann's theorem, which corresponds to the case $\alpha = 0$. Moreover, we can formulate an immediate corollary.

\begin{cor}
Assume that $(X,d)$ is a compact metric space and $(x_n)$ is a sequence of elements of $X$ with the set of limit points $A$. Then there exists a permutation $\sigma\colon \N \to \N$ such that $$\lim_{n \to \infty} d(x_n,x_{\sigma(n)})=\alpha$$   if and only if $\alpha \in Z(A)$. 
\end{cor}

The aim of this note is to show an application of the $c_0$-equivalence in the case of unbounded sequences. Our main theorem can be treated as an extension of the Riemann rearrangement theorem.

\section{Main result}

For an arbitrary number $a\in \R$ we denote $a^+=\max\{ a, 0\}$ and $a^-=a^{+}-a$. A series $\sum\limits_{n=1}^\infty a_n$ is called \emph{potentially conditionally convergent} if the following three conditions hold:
\begin{enumerate}
	\item[(a)] $\sum\limits_{n=1}^\infty a_n^+ = + \infty $,
	\item[(b)] $\sum\limits_{n=1}^\infty a_n^- = + \infty $,
	\item[(c)] $\lim\limits_{n \to \infty} a_n = 0$.
\end{enumerate}
It is clear that every conditionally convergent series is potentially conditionally convergent.
	
Our main theorem reads.

\begin{thm}\label{t}
Let $(a_n)$ be a sequence of real numbers such that the series $\sum\limits_{n=1}^\infty a_n$ is convergent and assume that there exists a set $A\subseteq \N$ such that the series  $\sum\limits_{n \in A} a_n$ is conditionally convergent. Next, let $(b_n)$ be an arbitrary sequence of real numbers. Then there exists a permutation $\sigma\colon \N \to \N$ such that  $\sigma(n) = n$ for every $n \notin A$ and the sequence $(b_n)$ is $c_0$-equivalent to some subsequence of the sequence of partial sums of the series $\sum\limits_{n=1}^\infty a_{\sigma(n)}$.
\end{thm}
\begin{proof}
Note that the series  $\sum\limits_{n=1}^\infty a_n$ is conditionally convergent. In particular, we have $\sum\limits_{n=1}^\infty a_n^+ = + \infty $ and
$\sum\limits_{n=1}^\infty a_n^- = + \infty $.  

We will inductively construct a bijection $\sigma\colon \N \to \N$  and a strictly increasing sequence of positive integers $(k_n)$ such that $S_{k_n}-b_n$ tends to $0$ as $n \to \infty$, where $S_{k_n}=\sum\limits_{j=1}^{k_n}a_{\sigma(j)}$. 

First, we put $\sigma(j) = j$ for all $j \notin A$. {Let us put} $k_0 =0$ and observe that for $n=0$ {the inductive assumptions are trivially satisfied, so we have the base for our induction}.

If $n>0$ let us assume that we have found positive integers $k_1, k_2, \dots , k_n$ which form a strictly increasing sequence, the values $\sigma(j)$ are defined for every integer $j\in\{1, \dots, k_n\}$, one has $\sigma(j) \in A$ whenever $j \in A$ and 
\begin{equation}\label{Sm}
\left|S_{k_m}-b_m\right|\leq \frac{1}{m}, \quad {for} \quad m \leq n.
\end{equation} 
If $n=0$, then we do not assume that \eqref{Sm} holds and moreover, any value of $\sigma$ is not defined at the moment. All subsequent arguments remain the same.

In case $n>0$ for $m<n$ let $k^*_m$ be the smallest integer which is greater than $k_m$ and belongs to $A$. {Let us assume moreover that} $\sigma(k_m^*)\in A$ is the smallest positive integer which does not belong to the set $\{\sigma(j) : j\leq k_m\}$. {This assumption will be used to obtain the surjectivity of} $\sigma$.

Since the series $\sum\limits_{j \notin A} a_j$ is convergent, then it satisfies the Cauchy condition. Therefore there exists a positive integer $k'_{n+1}>k_n$ such that for $q > k'_{n+1}$ one has
\begin{equation}
\left| \sum\limits_{j=1  , \, j \notin A}^{q} a_j - \sum\limits_{j=1  , \, j \notin A}^{k'_{n+1}} a_j  \right|   = \left|  \sum\limits_{j=k'_{n+1} +1  , \, j \notin A}^{q} a_j    \right|   \leq \frac{1}{2(n+1)}.\label{n+1}
\end{equation}
Now, we are able to define a positive integer $k_{n+1}>k'_{n+1}$  and values $\sigma(j)$ for $k_n+1\leq j \leq k_{n+1}$ in such a way that the following two conditions hold:
\begin{enumerate}
	\item[(A)] $\sigma(k^*_n)$ is the smallest positive integer in $A$ which does not appear in the set  $\{\sigma(j) : j\leq k_n,\, j \in A\}$,
	\item[(B)] the following estimate holds true:
\end{enumerate}
\begin{equation}\label{bn}
\left|  \sum\limits_{j=k_n+1  , \, j \in A}^{k_{n+1}} a_{\sigma(j)}-\left( b_{n+1} - \sum\limits_{j=1}^{k_{n}} a_{\sigma(j)}-
 \sum\limits_{j=k_n+1  , \, j \notin A}^{k'_{n+1}} a_j   \right)   \right|     \leq \frac{1}{2(n+1)}.
\end{equation}

{Indeed, the series} $\sum\limits_{j \in A} a_j$ {is conditionally convergent. Hence any real number} $r$ {can be estimated by the partial sums of a series} $\sum\limits_{j \in A} a_{\sigma_{1}(j)}$ {for some permutation} $\sigma_1 \colon \N \to \N$. {Obviously, we can take} $\sigma_1(j)=\sigma(j)$ for $j \leq k_n$. So for some $k_{n+1}$ we can have
$$ \left|  \sum\limits_{j=1  , \, j \in A}^{k_{n+1}} a_{\sigma_1(j)}-\left( b_{n+1} - \sum\limits_{j=1  , \, j \notin A}^{k'_{n+1} } a_{j}  \right)   \right|     \leq \frac{1}{2(n+1)}, $$
what, defining $\sigma(j):= \sigma_1(j)$ for $k_n+1 \leq j \leq k_{n+1}$ {gives us \eqref{bn}. Of course we can assume that} $k_{n+1}$ {is greater than} $k'_{n+1}$. {Making use of \eqref{n+1} and \eqref{bn} we obtain}
\begin{align*}
\left|S_{k_{n+1}}-b_{n+1}\right|
&=\left| \sum\limits_{j=1}^{k_{n+1}} a_{\sigma(j)} -  b_{n+1}\right| \\&= \left| \sum\limits_{j=1}^{k_{n}} a_{\sigma(j)} +   \sum\limits_{j=k_n+1  , \, j \in A}^{k_{n+1}} a_{\sigma(j)}   +  \sum\limits_{j=k_n+1  , \, j \notin A}^{k_{n+1}} a_{j}  -  b_{n+1}\right| 
\\&= \left| \sum\limits_{j=1}^{k_{n}} a_{\sigma(j)} +   \sum\limits_{j=k_n+1  , \, j \in A}^{k_{n+1}} a_{\sigma(j)}   +  
\sum\limits_{j=k_n+1  , \, j \notin A}^{k_{n+1}} a_{j} - \right. \\ &\qquad - \left. \sum\limits_{j=k_n+1  , \, j \notin A}^{k'_{n+1}} a_{j} + \sum\limits_{j=k_n+1  , \, j \notin A}^{k'_{n+1}} a_{j} -  b_{n+1}\right|
\\&\leq \left| \sum\limits_{j=k_n+1  , \, j \in A}^{k_{n+1}} a_{\sigma(j)}   -\(b_{n+1} - \sum\limits_{j=1}^{k_{n}} a_{\sigma(j)}  - \sum\limits_{j=k_n+1  , \, j \notin A}^{k'_{n+1}} a_{j}  \) \right| \\&+   
\left|
\sum\limits_{j=k'_{n+1}+1  , \, j \notin A}^{k_{n+1}} a_{j}           
 \right| \leq  \frac{1}{2(n+1)} + \frac{1}{2(n+1)}  = \frac{1}{(n+1)}.
\end{align*}

Finally, let us observe that our definition of the value $\sigma(k_n^*)$ guarantees the surjectivity of $\sigma$.
\end{proof}

Let us formulate some corollaries.

\begin{cor}\label{c1}
Let $(a_n)$ be a sequence of real numbers such that the series $\sum\limits_{n=1}^\infty a_n$ is potentially conditionally convergent. Next, let $(b_n)$ be an arbitrary sequence of real numbers. Then there exists a permutation $\sigma\colon \N \to \N$ such that the sequence  $(b_n)$ is $c_0$-equivalent to some subsequence of the sequence of partial sums of the series $\sum\limits_{n=1}^\infty a_{\sigma(n)}$.
\end{cor}
\begin{proof}
There exists a permutation $\pi\colon\N\to\N$ such that the series $\sum\limits_{n=1}^\infty a_\pi(n)$ is convergent. Apply our theorem for this series and the set $A=\N$ to get a permutation $\sigma'\colon\N\to\N$ such that the sequence  $(b_n)$ is $c_0$-equivalent to some subsequence of the sequence of partial sums of the series $\sum\limits_{n=1}^\infty a_{(\sigma'\circ \pi)(n)}$. To finish the proof take $\sigma = \sigma'\circ \pi$.
\end{proof}

W. Wilczy\'nski \cite{Wi} has proved that for every conditionally convergent series $\sum\limits_{n=1}^\infty a_n$ there exists a set $A$ of density zero such that the series $\sum\limits_{n=1, n \in A}^\infty a_n$ is conditionally convergent. He posed an open problem in \cite{Wi} to characterize all set ideals $\I$ with the property that for every conditionally convergent series $\sum\limits_{n=1}^\infty a_n$ there exists a set $A \in \I$ such that the series $\sum\limits_{n=1, n \in A}^\infty a_n$ is conditionally convergent.  R. Filip\'{o}w and P. Szuca \cite{FSz} answered this problem and they proved that a set ideal $\I$ has the above-mentioned property if and only if it cannot be extended to a summable ideal. 

\medskip
 
On joining \cite{Wi}*{Lemma} with our main result, we can take the set $A$ to be of density zero.

\begin{cor}
Let $(a_n)$ be a sequence of real numbers such that the series $\sum\limits_{n=1}^\infty a_n$ is conditionally convergent. Next, let $(b_n)$ be an arbitrary sequence of real numbers. Then there exists a permutation $\sigma\colon \N \to \N$ such that  
$$\lim_{n\to \infty}\frac{1}{n}|\{n \in \N : \sigma(n) \neq n\}| = 0$$ and the sequence $(b_n)$ is $c_0$-equivalent to some subsequence of the sequence of partial sums of the series $\sum\limits_{n=1}^\infty a_{\sigma(n)}$.
\end{cor}

It is worth to note a connection of our results with the Riemann rearrangement theorem. One can reformulate this classical result as follows:

\begin{thm}[Riemann rearrangement theorem]
Let $(a_n)$ be a sequence of real numbers such that the series $\sum\limits_{n=1}^\infty a_n$ is conditionally convergent, let $b\in \R\cup\{\pm\infty\}$ be arbitrary and let  $(b_n)$ be a sequence of real numbers tending to $b$. Then there exists a permutation $\sigma\colon \N \to \N$ such that  the sequence $(b_n)$ is $c_0$-equivalent to the sequence of partial sums of the series $\sum\limits_{n=1}^\infty a_{\sigma(n)}$.
\end{thm}

It is easy to note two differences between our Corollary \ref{c1} and the Riemann rearrangement theorem. First, we do not assume that the sequence $(b_n)$  is convergent (in fact, it can be even unbounded in our settings). But in return for it, in general only a subsequence of the sequence of partial sums of the series $\sum\limits_{n=1}^\infty a_{\sigma(n)}$ is $c_0$-equivalent to $(b_n)$.

\medskip

Recently several results related to the Riemann rearrangement theorem and conditional convergence of series were published by several authors, see \cites{Ba, Ch, DS, FF, G, M, P, W1, W2, W3}.

\medskip

We will terminate the paper with a generalization of an interesting result of Burgin \cite{B0}*{Corollary 4.6} which extends the classical Riemann rearrangement theorem to the case of the so-called hypernumbers. However, the original proof of Bounded Riemann Series Theorem (BRST, \cite{B0}*{Theorem 4.5}) contains a gap which cannot be removed directly. We give some counterexamples to this statement. We provide an alternative argument and at the same time we obtain a slightly more general statement than the Burgin's Generalized Riemann Series Theorem (GRST, \cite{B0}*{Corollary 4.6}).

Hypernumbers have been founded by Burgin and they provide an extension of the real line. This extension is of importance in many areas of mathematics and in physics. The elegant and deep theory of hypernumbers allows one in particular to differentiate in an extended sense every real function.
We recall here only basic notions we need in what follows. For a comprehensive study of the topic the reader is referred to the excellent monograph by Burgin \cite{B1}.
Let $\R^\N$ stand for the set of all sequences of real numbers. Following Burgin's notation, the set of $c_0$-equivalence classes of  $\R^\N$ is denoted by $\R_\omega$ and its elements are called \emph{hypernumbers}. Given a sequence $(a_n)\in \R^\N$, by $\Hn( (a_n))$ we denote its equivalence class, i.e. the hypernumber which is represented by this sequence. In particular, every real number can be identified with a constant sequence equal to it and therefore the real line $\R$ is a proper subset of the set of all hypernumbers $\R_\omega$. A hypernumber is called \emph{bounded} if it is generated by a bounded sequence. If $\alpha, \beta \in \R_\omega$, then we say that $\beta$ is a \emph{subnumber} of $\alpha$, which is denoted by $ \beta \Subset \alpha $, if there exist two sequences $(a_n), (b_n)$ such that $(b_n)$ is a subsequence of $(a_n)$, $\alpha = \Hn( (a_n))$ and $\beta = \Hn( (b_n))$.  

Following Burgin \cite{B0}, given a series of real numbers $\sum_{n=1}^\infty a_n$, convergent or not, we say that a series  $\sum\limits_{n=1}^\infty b_n$ is a \emph{quotient series} of $\sum\limits_{n=1}^\infty a_n$ if there is a surjective map $p\colon \N \to \N$ such that for every $i \in \N$ the set $p^{-1}(i)$ consists of a finite number of consecutive elements of $\N$ and $b_i=\sum\limits_{j \in p^{-1}(i)}a_j$. If additionally the map $p$ is monotone, then we speak about a \emph{monotone quotient series}.

Burgin \cite{B0} introduced the concept of \emph{analytical sums} of a series as a hypernumber generated by its partial sums. For an arbitrary series of real numbers $\sum\limits_{n=1}^\infty a_n$ we write $\mathrm{an}(\sum\limits_{n=1}^\infty a_n) = \Hn((A_n))$ , where $A_n=\sum\limits_{i=1}^n a_i$ for $n \in \N$.
If the series $\sum\limits_{n=1}^\infty a_n$ is convergent, then its analytical sum coincides with its topological sum, i.e. the real number equal to its sum, which is denoted by $\mathrm{top}(\sum\limits_{n=1}^\infty a_n)$. For series divergent to $\pm \infty$ one can also speak about its topological sums. 

One can observe that if $\alpha, \beta \in \R_\omega$ satisfy $$\alpha = \mathrm{an}\(\sum\limits_{n=1}^\infty a_n \), \quad \beta = \mathrm{an}\(\sum\limits_{n=1}^\infty b_n\)$$ for some $(a_n), (b_n) \in \R^\N$, then the series  $\sum\limits_{n=1}^\infty b_n$ is a monotone quotient series of $\sum\limits_{n=1}^\infty a_n$ if and only if $ \beta \Subset \alpha $ (see \cite{B0}*{Lemma 4.4}).


The Bounded Riemann Series Theorem (BRST) \cite{B0}*{Theorem 4.5} deals with a series $\sum\limits_{n=1}^\infty a_n$ which is bounded and not absolutely convergent. Thus $\mathrm{an}(\sum\limits_{n=1}^\infty a_n)$ is a bounded hypernumber. BRST assets that for every hypernumber $\alpha \in \R_\omega$ there exists a monotone quotient series $\sum\limits_{n=1}^\infty d_n$ of $\sum\limits_{n=1}^\infty a_n$, a permutation $\sum_{n=1}^\infty b_n$ of $\sum\limits_{n=1}^\infty d_n$ and a monotone quotient series $\sum\limits_{n=1}^\infty c_n$ of $\sum\limits_{n=1}^\infty b_n$ such that $\alpha = \mathrm{an}(\sum\limits_{n=1}^\infty c_n)$. This statement is in general not true, which is exhibited by examples below.
Moreover, an inspection of the proof of \cite{B0}*{Theorem 4.5} allows us to detect an invalid argument that was used. Namely, it is not true that a monotone quotient series of a series which is not absolutely convergent is necessarily not absolutely convergent. Consequently, Riemann rearrangement theorem cannot be applied. Our examples shed some more light on the situation.

\begin{ex}
Let $a_n=(-1)^{n+1}$ for $n \in \N$. Then the series $\sum\limits_{n=1}^\infty a_n$  is not absolutely convergent (in fact it is divergent) and its analytical sum is a bounded hypernumber. Thus the assumptions of BRST are satisfied. Note that every converging subsequence of its partial sums $A_n$ must be from some point equal either to $0$ or $1$. Therefore, every monotone quotient series $\sum\limits_{n=1}^\infty d_n$ of $\sum\limits_{n=1}^\infty a_n$ which is convergent, converges to $0$ or $1$. 
Thus, every permutation $\sum\limits_{n=1}^\infty b_n$ of $\sum\limits_{n=1}^\infty d_n$ is either equal to the zero series, or to a series which contains $1$ at one place. 
Consequently, every monotone quotient series $\sum\limits_{n=1}^\infty c_n$ of $\sum\limits_{n=1}^\infty b_n$ is also either equal to the zero series, or it contains $1$ at one place.
Two in a sense typical examples of monotone quotient series of $\sum\limits_{n=1}^\infty a_n$  are $0+0+ \dots $ and  $1+0+0+ \dots $. 
The first case can be obtained by the projection $j_0\colon \N \to \N$ given by $j_0(2k-1)=j_0(2k) = k$ for $k\in \N$, whereas the second one by the projection $j_1\colon \N \to \N$ given by $j_1(1) = 1$ and $j_1(2k)=j_1(2k+1) = k$ for $k\in \N$. 
Thus, only two hypernumbers $\alpha$, namely $0$ and $1$ can be achieved from the series $\sum\limits_{n=1}^\infty a_n$ by means of the proof of Burgin's BRST. Note however that every hypernumber which is of the form $\Hn (q_n)$ with all $q_n$ being integers, can be obtained as a monotone quotient series of a permutation of $\sum\limits_{n=1}^\infty a_n$. Moreover, any hypernumber which is not of this form cannot be obtained as a monotone quotient series of $\sum\limits_{n=1}^\infty a_n$. 
\end{ex}

\begin{ex}
Let $a_n=(-1)^{n+1}/n$ for $n \in \N$. Then the series $\sum\limits_{n=1}^\infty a_n$  is conditionally convergent (to $\ln 2$). Again, the assumptions of BRST are satisfied. This time there exists a monotone quotient series $\sum\limits_{n=1}^\infty d_n$ of $\sum\limits_{n=1}^\infty a_n$ which is absolutely convergent. Indeed, one can take any projection from the previous example, if we take  $j_0\colon \N \to \N$ given by $j_0(2k-1)=j_0(2k) = k$ for $k\in \N$, then we obtain a convergent series of positive numbers. But in contrast to the previous example, this time it is possible to achieve every hypernumber as a monotone quotient series of a permutation of the original one. This is a straightforward consequence of Corollary \ref{cf} below. 
\end{ex}

\begin{cor}\label{cf}
Assume that $\alpha \in \R_\omega$ is an arbitrary hypernumber and  $(a_n)$ is  a sequence of real numbers such that the series $\sum\limits_{n=1}^\infty a_n$ is potentially conditionally convergent.
Then, there exists a permutation $\sigma\colon \N\to\N$ such that $\alpha$ is equal to a monotone quotient series of the series $\sum\limits_{n=1}^\infty a_{\sigma(n)}$.
\end{cor}
\begin{proof}
Follows immediately from our Theorem \ref{t}.
\end{proof}

\end{document}